\documentclass[oneside, 12pt]{amsart}

\usepackage[utf8]{inputenc}
\usepackage[margin=1in]{geometry}
\usepackage{amsthm}
\usepackage{enumitem}  
\usepackage[OT2,T1]{fontenc}
\usepackage{amscd, amssymb, enumitem, amsmath, mathrsfs, amsfonts}
\usepackage{url}
\usepackage{tikz}
\usepackage{fullpage}
\usepackage{microtype}
\usetikzlibrary{decorations.pathreplacing}
\usepackage[backref=page]{hyperref}
\usepackage{hyperref}
\usepackage[utf8]{inputenc}
\usepackage{comment}

\newcommand{\widebar}[1]{\mkern 1.5mu\overline{\mkern-1.5mu#1\mkern-1.5mu}\mkern 1.5mu}

\usepackage[utf8]{inputenc}

\newtheorem{theorem}{Theorem}[section]
\newtheorem{lemma}[theorem]{Lemma}
\newtheorem{proposition}[theorem]{Proposition}
\newtheorem*{proposition*}{Proposition}

\newtheorem*{questiona*}{Question A}
\newtheorem*{questionb*}{Question B}
\newtheorem*{theorem*}{Theorem}
\newtheorem*{question*}{Question}
\newtheorem*{question}{Question 1}

\theoremstyle{definition}
\newtheorem{definition}[theorem]{Definition}
\newtheorem{example}[theorem]{Example}

\newtheorem{remark}[theorem]{Remark}

\theoremstyle{remark}

\newcommand{\Z}{\mathbb{Z}}
\newcommand{\Q}{\mathbb{Q}}

\newcommand{\ve}{\varepsilon}
\newcommand{\co}{\mathcal{O}}

\title{Universal Quadratic forms and Dedekind Zeta functions}
\author{Vítězslav Kala}
\author{Mentzelos Melistas}

\address{Charles University, Faculty of Mathematics and Physics, Department of
Algebra, Sokolov\-sk\' a 83, 18600 Praha~8, Czech Republic}

\email[V.~Kala]{vitezslav.kala@matfyz.cuni.cz}

\address{Charles University, Faculty of Mathematics and Physics, Department of
Algebra, Sokolov\-sk\' a 83, 18600 Praha~8, Czech Republic}

\address{University of Twente, Department of Applied Mathematics, Drienerlolaan 5, 7522 NB Enschede, The Netherlands}
\email[M.~Melistas]{mentzmel@gmail.com}
\date{\today}

\thanks{
The authors were supported by Czech Science Foundation (GA\v CR) grant 21-00420M}

\begin{document}

\maketitle

\begin{abstract}
    We study universal quadratic forms over totally real number fields using Dedekind zeta functions. In particular, we prove an explicit upper bound for the rank of universal quadratic forms over a given number field $K$, under the assumption that the codifferent of $K$ is generated by a totally positive element. Motivated by a possible path to remove that assumption, we also investigate the smallest number of generators for the positive part of ideals in totally real numbers fields.
\end{abstract}

\section{Introduction}

The study of universal quadratic forms, i.e., positive definite quadratic forms that represent all natural numbers, has a rich and very long history, involving works of mathematicians such Diophantus, Brahmagupta, Fermat, Euler, and Gauss. Among the most important contributions of the modern era to the subject are the Conway--Schneeberger theorem (also called $15$-theorem) \cite{bhargava} and the $290$-theorem of Bhargava--Hanke \cite{bhargavahanke}.

A natural generalization is to replace $\mathbb{Z}$ with the ring of integers $\mathcal{O}_K$ of a totally real number field $K$. In this more general case a quadratic form is called universal if it is totally positive definite and represents all totally positive elements of $\mathcal{O}_K$. In 1941 Maa{\ss} \cite{mass} showed that the sum of three squares is universal over ${\mathbb{Q}(\sqrt{5})}$. Moreover, Siegel \cite{siegel1945} in 1945 proved that if the sum of any number of squares is universal over $K$, then $K=\mathbb{Q}$ or $\mathbb{Q}(\sqrt{5})$. On the other hand, the result of  Hsia, Kitaoka, and Kneser \cite{hkk} implies that there exists a universal quadratic form over every number field. However, the question of finding the minimal number of variables needed for a universal form in a given number field still remains open and hard.

Blomer--Kala \cite{BK}, \cite{Ka1} showed that, given a positive number $N$, there exists a real quadratic field such that any universal quadratic form over it must have rank at least $N$, i.e., it must have at least $N$ variables. In this article we are interested in extracting more precise information concerning ranks of universal quadratic forms over a number field $K$ using its Dedekind zeta function $\zeta_K$. Strategies and techniques employing the Dedekind zeta have already been used by Yatsyna \cite{yatsyna2019} as well as by Kala--Yatsyna  \cite{kalayatsyna}.

Before we state our main results let us introduce some notation that will be needed below. First, if $K$ is a number field with ring of integers $\mathcal{O}_K$, then we will denote by $\mathcal{O}_K^{\vee}$ the codifferent of $K$ (see Section \ref{section2} below for any undefined notions). If $d$ is a positive integer, then we define
\begin{equation}
    r_d = \begin{cases}\label{eqr} 
          \left\lfloor \frac{d}{6} \right\rfloor &  \text{if } d \equiv 1 \; (\text{mod }6),  \\
          \left\lfloor \frac{d}{6} \right\rfloor +1  & \text{if } d \not\equiv 1 \; (\text{mod }6).  \\ 
       \end{cases}
\end{equation}

Moreover, if $K$ has degree $d$ and discriminant $\Delta_K$, then we write $$g(\ell, \Delta_K)=\max \left\{ 4, e^{\gamma_0}\frac{\ell^d}{d^d} \Delta_K \log \log \left(\frac{\ell^d}{d^d} \Delta_K\right) + \frac{0.6483\frac{\ell^d}{d^d} \Delta_K}{\log \log (\frac{\ell^d}{d^d}\Delta_K) } \right\},$$
where $\gamma_0$ is the Euler-Mascheroni constant.

In Section \ref{section2}, by adapting and slightly generalizing a technique that was used by Kala--Yatsyna \cite[Proof of Theorem 5.1]{kalayatsyna}, we prove the following theorem.

\begin{theorem}\label{maintheorem}
    Let $K$ be a totally real number field of degree $d$. Assume that $\mathcal{O}^{\vee}_K=(\delta)$ with $\delta$ totally positive and denote by $r_d$ the integer of Equation \eqref{eqr}. If $Q$ is a universal form of rank $R$ over $\mathcal{O}_K$, then 
    $$2\binom{Rd+4r_d-1}{4r_d-1}-1 >  \frac{G(\Delta_K) }{B(d) 2^d} |\Delta_K|^{\frac{3}{2}}\left(\frac{1}{4 \pi}\right)^d ,$$
    where $$G(\Delta_K)=\min \left\{ \frac{1}{g(\ell, \Delta_K )} \: : \: \ell \leq r_d \right\}$$ and $$ B(d)= \begin{cases} 
          \min \{ \frac{1}{b_{\ell}(2d)} \: : \: \ell \leq r_d \text{ and } b_{\ell}(2d) > 0 \}, & \text {if } d \text{ is even}, \\
          \min \{ \frac{1}{-b_{\ell}(2d)} \: : \: \ell \leq r_d \text{ and } b_{\ell}(2d) < 0 \}, & \text{if } d \text{ is odd},
       \end{cases}$$ for some rational numbers $b_1(2d),\dots,b_{\ell}(2d)$ that only depend on $d$.
\end{theorem}

As a corollary of Theorem \ref{maintheorem}, we obtain the following theorem, which slightly generalizes a theorem of Yatsyna \cite[Theorem 4]{yatsyna2019}.

\begin{theorem}\label{theoremcorollary}
    Let $d$ and $R$ be positive integers. Then there exist only finitely many totally real number fields of degree $d$ with codifferent generated by a totally positive element that have a universal quadratic form of rank $R$ defined over them.
\end{theorem}

We employ quadratic lattices in the proofs of both of these theorems and, in fact, they hold also for universal quadratic lattices.

We also obtain (in Theorem \ref{theoremliftingproblem}) a result on the lifting problem \cite{kalayatsyna, kalayatsynabulletinofthelms, kimlee}, i.e., on the question of universality  of quadratic forms whose coefficients lie in $\Z$.

One may wonder whether the assumption that $\mathcal{O}^{\vee}_K=(\delta)$ in Theorem \ref{maintheorem} is really needed. Motivated by this consideration, in Section \ref{section3}, we are led to the study of the number of generators for the positive part of ideals in rings of integers of number fields, which may be of independent interest. To be more precise, let $K$ be a totally real number field with ring of integers $\mathcal{O}_K$.
For a subset $J$ of $K$, we denote by $J^+$ the set of all totally positive elements of $J$, and 
$J^{+,0}=J^+ \cup \{ 0\}$. 
One can then ask the following question.

\begin{question}\label{question1} Let $I$ be a fractional ideal in $K$.
What is the smallest number $n$ such that there exist $\alpha_1, \alpha_2,\dots,\alpha_n \in K^+$ with $$I^{+,0}=\alpha_1 \mathcal{O}_K^{+, 0}+\dots+ \alpha_n\mathcal{O}_K^{+, 0}?$$
\end{question}

In the last section, using the notion of $I$-indecomposables, we show that Question \ref{question1} above is well posed in the sense that given any ideal $I$ of $K$,  $I^{+,0}$ is generated by a finite number of totally positive elements. Consequently, given any ideal $I$, we can define $\kappa(I)$ to be the smallest number of generators for $I^{+,0}$ and we can also define  $\kappa(K)=\max_{I} \kappa(I)$ (where the maximum is taken over all the fractional ideals $I$, see Definition \ref{defkappa} below). Using the narrow class group of $K$ we show, in Proposition \ref{propositionkappa}, that $\kappa(K)$ is finite. After that, we focus on the case where $K$ is a real quadratic field and we prove a result relating $\kappa(I)$ to continued fractions.

\section{Siegel's Formula}\label{section2}

In this section, we use Siegel's formula to study number fields that possess universal quadratic forms and we prove Theorem \ref{maintheorem}. Before we proceed let us recall some background material and fix the notation that we will use. 
Let $K$ be a totally real number field with ring of integers $\mathcal{O}_K$; let us denote the norm and trace maps from $K$ to $\Q$ by $\mathrm N$ and $\mathrm{Tr}$.
The codifferent of $K$ is defined by $$\mathcal{O}_K^{\vee}=\{ \alpha \in K : \mathrm{Tr}(\alpha \mathcal{O}_K) \subseteq \mathbb{Z}\}.$$ 
For a subset $J$ of $K$, we denote by $J^+$ the set of all totally positive elements of $J$, and 
$J^{+,0}=J^+ \cup \{ 0\}$. 

We will use the language of lattices so that we can work with quadratic forms in a coordinate-free way. We refer the reader to \cite{omearabook} for general background material on lattices and quadratic forms. An $r$-ary quadratic $K$-space $V$ is an $r$-dimensional vector space $V$ over $K$ equipped with a symmetric bilinear form $B: V \times V \rightarrow K$. The associated quadratic form $Q : V \rightarrow K$ is given by $Q(x)=B(x,x)$ for every $x \in V.$ The Gram matrix of vectors $v_1,\dots,v_n \in V$ is the matrix $(B(v_i,v_j)) \in M_{n\times n}(K)$.

Let $V$ be an $r$-ary quadratic space over $K$. A quadratic $\mathcal{O}_K$-lattice $L$ of rank $r$ is a finitely generated $\mathcal{O}_K$-submodule of $V$ such that $KL=V$ equipped with the restrictions of $B$ and $Q$ to $L$. The rank of $L$ is $r$ and is denoted by $\text{rank}(L)$. A quadratic $\mathcal{O}_K$-lattice $L$ is called (totally) positive definite if $Q(x) \in \mathcal{O}^+_K$ for every $0\neq x \in L$. 
By a $\Z$-form we mean a positive definite quadratic form with $\Z$-coefficients, naturally considered as a quadratic form on the $\mathcal{O}_K$-lattice $\mathcal{O}_K^r$.

Finally, a positive definite lattice $L$ is called universal if every element of $\mathcal{O}_K^+$ is represented by an element of $L$, i.e., for every $w \in \mathcal{O}_K^+$ there exists $x \in L$ such that $Q(x)=w$.

We will need the following upper bound on the number of short vectors in a given $\Z$-lattice, due to Regev--Stephens-Davidowitz.

\begin{theorem}\label{theoremcin}\cite[Theorem 1.1]{rsd}
    Let $L$ be a positive definite $\mathbb{Z}$-lattice of rank $R$. If $i$ is a positive integer, then we denote by $N_{\leq i}(L)$ the number of vectors of norm $\leq i$ in $L$. Then $$N_{\leq i}(L) \leq C(R,i)=2\binom{R+4i-1}{4i-1}-1.$$
\end{theorem}

Note that \cite[Theorem 1.1]{rsd} concerns only classical quadratic lattices, whereas here we state the result without the classical assumption. The transition is easy: If $(L,Q)$ is non-classical, then $(L, 2Q)$ is classical, and vectors with $Q(v)\leq i$ correspond to vectors with $2Q(v)\leq 2i$, giving the version of the theorem formulated above.
Note also that the bound for vectors of norm 2 (in classical lattices) can be improved, see, e.g., \cite[Theorem 3.1]{kalayatsynaZmija} or \cite[Section 4]{rsd}.

Let now $K$ be a totally real number field of degree $d>1$. We will work with its Dedekind zeta function $\zeta_K(s)$, defined as the meromorphic function that is for $s\in\mathbb C, \mathrm{Re}(s)>1,$ given by
$$\zeta_K(s)=\sum_{I<\mathcal O_K} \frac 1 {\mathrm N(I)^s},$$
where the sum runs over all the non-zero integral ideals $I$ in $\mathcal O_K$, and $\mathrm N(I)$ denotes the norm of an ideal $I$. As is well-known, $\zeta_K$ satisfies a functional equation, from which we will need that
\begin{equation}\label{eq:zeta FE}
    \zeta_K(-1)=(-1)^d |\Delta_K|^{\frac{3}{2}}\left(\frac{1}{4 \pi}\right)^d \zeta_K(2).
\end{equation}

Further, for any (integral) ideal $I$ of $\mathcal{O}_K$ we define $$\sigma (I)= \sum_{J \mid I} \mathrm{N}(J).$$ Moreover, we define
\begin{equation}\label{equationsk}
    s_{\ell}^K(2)=\sum_{\gamma \in \mathcal{O}^{\vee,+}_K, \;  \mathrm{Tr}(\gamma)=\ell } \sigma ((\gamma) (\mathcal{O}^{\vee}_K)^{-1}).
\end{equation}
Then we have the following theorem due to Siegel. 

\begin{theorem}\label{theoremzetafuntion}\cite[Page 59]{zagierzetafunction}
Let $K$ be a totally real number field of degree $d>1$. %Let $h=2dm$, with $m$ a positive integer. 
Then 
$$\zeta_K(-1)=2^d \sum_{\ell=1}^{r_d} b_\ell(2d) s_{\ell}^K(2),$$
where $b_1(2d),\dots,b_r(2d)$ are rational numbers that only depend on $d$, and $r_d, s_{\ell}^K(2)$ are given in Equations \eqref{eqr} and \eqref{equationsk}.

\end{theorem}

Before we proceed to the proof of Theorem \ref{maintheorem}, we need to provide bounds for the terms $\sigma ((\gamma) (\mathcal{O}^{\vee}_K)^{-1})$ involved in Equation $($\ref{equationsk}$)$. We do this in the following lemma.

\begin{lemma}\label{lemmasigma}
    Let $K$ be a totally real number field of degree $d>1$. Assume that $\mathcal{O}^{\vee}_K=(\delta)$ with $\delta$ totally positive. Let $\ell$ be a positive integer and let $\gamma \in \mathcal{O}^{\vee, +}_K$ such that $\mathrm{Tr}(\gamma)=\ell$. Then $$\sigma ((\gamma) (\mathcal{O}^{\vee}_K)^{-1}) \leq g(\ell, \Delta_K).$$
\end{lemma}
\begin{proof} 
    Since $\delta$ is totally positive, we find that $\mathcal{O}^{\vee, +}_K=\delta\mathcal{O}^{+}_K$. As $\gamma \in \mathcal{O}^{\vee, +}_K$, there exists $\alpha \in \mathcal{O}_K^+$ such that $\gamma= \delta \alpha$. Therefore,
    $$\sigma ((\gamma) (\mathcal{O}^{\vee}_K)^{-1})=\sigma ((\alpha)(\delta) (\mathcal{O}^{\vee}_K)^{-1})=\sigma ((\alpha) \mathcal{O}^{\vee}_K (\mathcal{O}^{\vee}_K)^{-1})=\sigma (\alpha \mathcal{O}_K)=\sum_{J \mid \alpha \mathcal{O}_K } \mathrm{N} (J).$$ 

    However, if $J$ is an ideal of $\mathcal{O}_K$ with $J \mid \alpha \mathcal{O}_K$, then $N(J)$ divides $ N(\alpha \mathcal{O}_K)$. Consequently, 
    $$\sum_{J \mid \alpha \mathcal{O}_K } \mathrm{N} (J)=\sum_{N(J) \mid N(\alpha \mathcal{O}_K) } \mathrm{N} (J) \leq \sum_{d \mid N(\alpha \mathcal{O}_K)} d= \sigma'(N(\alpha \mathcal{O}_K))=\sigma'(\mathrm{N}(\alpha)),$$ 
    where $\sigma'$ is the classical divisor function.  It is a theorem of Robin \cite[Théorème 2]{robinbound} that for $n \geq 3$ we have that $$\sigma'(n)<e^{\gamma_0}n \log \log (n) + \frac{0.6483n}{\log \log (n) },$$ where $\gamma_0$ is the Euler–Mascheroni constant. Using this we find $$\sigma ((\gamma) (\mathcal{O}^{\vee}_K)^{-1}) \leq e^{\gamma_0}\mathrm{N}(\alpha) \log \log (\mathrm{N}(\alpha)) + \frac{0.6483\mathrm{N}(\alpha)}{\log \log (\mathrm{N}(\alpha)) }.$$

    On the other hand, a simple application of the Arithmetic--Geometric means inequality gives 

    $$\ell =\mathrm{Tr}(\gamma) \geq d (\mathrm{N}(\gamma))^{\frac{1}{d}}=d (\mathrm{N}(\delta))^{\frac{1}{d}} (\mathrm{N}(\alpha))^{\frac{1}{d}}.$$
    Therefore, $$\mathrm{N}(\alpha) \leq \frac{\ell^d}{d^d} \frac{1}{\mathrm{N}(\delta)}.$$

    Since the different ideal $(\mathcal{O}_K^{\vee})^{-1}$ has norm equal to the discriminant, and $\frac{1}{\delta}$ is a generator for $(\mathcal{O}_K^{\vee})^{-1}$, we find that $\mathrm{N}(\delta)=\frac{1}{\Delta_K}$. Finally, since $$ e^{\gamma_0}\mathrm{N}(\alpha) \log \log (\mathrm{N}(\alpha)) + \frac{0.6483\mathrm{N}(\alpha)}{\log \log (\mathrm{N}(\alpha)) } \leq  e^{\gamma_0}\frac{\ell^d}{d^d} \Delta_K \log \log (\frac{\ell^d}{d^d} \Delta_K) + \frac{0.6483\frac{\ell^d}{d^d} \Delta_K}{\log \log (\frac{\ell^d}{d^d}\Delta_K) }$$ and we have that $\sigma'(2)=3$, $\sigma'(3)=4$, we find that $$\sigma ((\gamma) (\mathcal{O}^{\vee}_K)^{-1}) \leq g(\ell, \Delta_K).\qedhere$$     
\end{proof}

\begin{proof}[\it Proof of Theorem $\ref{maintheorem}$]
    Fix an integral basis $\omega_1, \omega_2,\dots, \omega_d$ for $\mathcal{O}_K$. Consider the quadratic form $q$ given by $$q(x_{11},x_{12},\dots,x_{Rd})=\mathrm{Tr}(\delta Q(x_{11}\omega_1+\dots+x_{1d}\omega_d,\dots,x_{R1}\omega_1+\dots+x_{Rd}\omega_d)),$$ for $x_{ij} \in \mathbb{Z}$. Since $\delta \in \mathcal{O}_K^{\vee,+}$ and $Q$ is positive definite, we see that $q$ is a positive definite quadratic form over $\mathbb{Z}$. By construction, we have that $q$ has rank $Rd$. 
    %Moreover, by replacing $\delta$ by $2\delta$ if $Q$ is not classical, we can also assume that $q$ is classical. Thus $(\mathbb{Z}^{Rd}, q)$ defines a classical positive definite $\mathbb{Z}$-lattice.

    For every positive integer $n$, let $N_{\leq i}(q)$ be the number of vectors $v \in \mathbb{Z}^{Rd}$ such that $q(v)\leq i$. It follows from Theorem \ref{theoremcin} that $N_{\leq i}(q)\leq  C(Rd,i)$. Therefore, we have the following inequality
   $$2\binom{Rd+4r_d-1}{4r_d-1}-1=C(Rd,r_d) \geq N_{\leq r_d}(q) = \# \{ w \in \mathbb{Z}^{Rd} \: : \: q(w) \leq r_d \}.$$

     Now, if $\gamma \in \mathcal{O}_K^{\vee,+}$, then $\gamma=\delta \alpha$ for some $\alpha \in \mathcal{O}^+_K$. Using the assumption that $Q$ is universal over $\mathcal{O}_K$, we find that there exists $(w_{11},\dots,w_{Rd}) \in \mathbb{Z}^{Rd}$ such that $$Q(w_{11} \omega_1+\dots+w_{1d}\omega_d,\dots,w_{R1}\omega_1+\dots+w_{Rd} \omega_d)=\alpha.$$ 
     This implies that $$\# \{ w \in \mathbb{Z}^{Rd} \: : \: q(w) \leq r_d \} \geq \# \{ \gamma \in \mathcal{O}^+_K \: : \: \mathrm{Tr}(\gamma) \leq r_d \} .$$ 
     %Here we really need $2r_d$ at the set of the left-hand side in the above inequality. Indeed, if $\gamma= \delta \alpha \in \mathcal{O}_K^{\vee,+}$ with $\alpha \in \mathcal{O}_K^+$ and $\mathrm{Tr}(\gamma) \leq r_d$, then, in the case where $Q$ is non-classical, we will have that $\mathrm{Tr}(2\delta \alpha)=q(w) \leq 2r_d$, with $\alpha=Q(w)$ for some $w \in \mathbb{Z}^{Rd}$.

     On the other hand, Lemma \ref{lemmasigma} implies that for every $\gamma \in \mathcal{O}_K^{\vee,+}$ with $\mathrm{Tr}(\gamma)=\ell$ we have the inequality $$1 \geq \frac{\sigma ((\gamma) (\mathcal{O}^{\vee}_K)^{-1})}{g(\ell, \Delta_K)}.$$ 
     
     Therefore, $$\# \{ \gamma \in \mathcal{O}^+_K \: : \: \mathrm{Tr}(\gamma) \leq r_d \} = \sum_{\ell=1}^{r_d} \sum_{\gamma \in \mathcal{O}^{\vee,+}_K, \;  \mathrm{Tr}(\gamma)=\ell } 1 \geq \sum_{\ell=1}^{r_d} \sum_{\gamma \in \mathcal{O}^{\vee,+}_K, \;  \mathrm{Tr}(\gamma)=\ell } \frac{\sigma ((\gamma) (\mathcal{O}^{\vee}_K)^{-1})}{g(\ell, \Delta_K)}.$$
    Let $$G(\Delta_K)=\min \left\{ \frac{1}{g(\ell, \Delta_K )} \: : \: \ell \leq r_d \right\} .$$  

    Assume first that $d$ is even. Recall that in this case, we set
    $$ B(d)=\min \left\{ \frac{1}{b_{\ell}(2d)} \: : \: \ell \leq r_d \text{ and } b_{\ell}(2d) > 0 \right\}.$$ 
    Since the sum $$\sum_{\gamma \in \mathcal{O}^{\vee,+}_K, \;  \mathrm{Tr}(\gamma)=\ell } \sigma ((\gamma) (\mathcal{O}^{\vee}_K)^{-1}) $$ is non-negative, we find the following inequality
    $$ \sum_{\ell=1}^{r_d} \sum_{\gamma \in \mathcal{O}^{\vee,+}_K, \;  \mathrm{Tr}(\gamma)=\ell } \frac{\sigma ((\gamma) (\mathcal{O}^{\vee}_K)^{-1})}{g(\ell, \Delta_K)} \geq \frac{G(\Delta_K)}{B(d)} \sum_{\ell=1}^{r_d} b_{\ell}(2d)  \sum_{\gamma \in \mathcal{O}^{\vee,+}_K, \;  \mathrm{Tr}(\gamma)=\ell } \sigma ((\gamma) (\mathcal{O}^{\vee}_K)^{-1}).$$ 
    Moreover, using Theorem \ref{theoremzetafuntion}, we see that the right-hand side of the above inequality is equal to 
    $$ \frac{G(\Delta_K)}{B(d) 2^d} 2^d \sum_{\ell=1}^{r_d} b_\ell(2d) s_{\ell}^K(2)=  \frac{G(\Delta_K) }{B(d) 2^d} \zeta_K(-1)\overset{\eqref{eq:zeta FE}}{=}
    (-1)^d \frac{G(\Delta_K) }{B(d) 2^d}|\Delta_K|^{\frac{3}{2}}\left(\frac{1}{4 \pi}\right)^d \zeta_K(2).$$ 
    As $\zeta_K(2) >1,$ by combining the inequalities obtained above, we see that our theorem is proved in the case of even $d$.
    
    Assume now that $d$ is odd.  Recall that in this case, we set
    $$ B(d)=\min \left\{ \frac{1}{-b_{\ell}(2d)} \: : \: \ell \leq r_d \text{ and } b_{\ell}(2d) < 0 \right\}.$$
     Since the sum $$\sum_{\gamma \in \mathcal{O}^{\vee,+}_K, \;  \mathrm{Tr}(\gamma)=\ell } \sigma ((\gamma) (\mathcal{O}^{\vee}_K)^{-1}) $$ is non-negative, we find the following inequality
    $$ \sum_{\ell=1}^{r_d} \sum_{\gamma \in \mathcal{O}^{\vee,+}_K, \;  \mathrm{Tr}(\gamma)=\ell } \frac{\sigma ((\gamma) (\mathcal{O}^{\vee}_K)^{-1})}{g(\ell, \Delta_K)} \geq -\frac{G(\Delta_K)}{B(d)} \sum_{\ell=1}^{r_d} b_{\ell}(2d)  \sum_{\gamma \in \mathcal{O}^{\vee,+}_K, \;  \mathrm{Tr}(\gamma)=\ell } \sigma ((\gamma) (\mathcal{O}^{\vee}_K)^{-1}).$$ 
    Therefore, using a similar argument as in the case where $d$ is even we find that the right-hand side of the above inequality is equal to $$ \frac{G(\Delta_K) }{B(d) 2^d} (-\zeta_K(-1))=\frac{G(\Delta_K) }{B(d) 2^d}|\Delta_K|^{\frac{3}{2}}\left(\frac{1}{4 \pi}\right)^d \zeta_K(2).$$ 
    This completes the proof of our theorem because we again have  $\zeta_K(2) >1$. 
\end{proof}

\begin{proof}[Proof of Theorem $\ref{theoremcorollary}$]
    The term $$\frac{G(\Delta_K) }{B(d) 2^d} |\Delta_K|^{\frac{3}{2}}\left(\frac{1}{4 \pi}\right)^d$$ grows to infinity as $|\Delta_K|$ grows to infinity (while everything else is fixed). Therefore, it follows from Theorem \ref{maintheorem} that given any fixed $R$ and $d$, there exist only a finite number of possibilities for the discriminant $\Delta_K$ of number fields of degree $d$ with codifferent generated by a totally positive element that have a universal quadratic form of rank $R$ defined over them. Consequently, by the Hermite--Minkowski theorem there only finitely many possibilities for such number fields. 
\end{proof}

We end this section by recording the following theorem, which is a more general version of a part of a theorem of Kala--Yatsyna \cite[Theorem 5.1]{kalayatsyna}.

\begin{theorem}\label{theoremliftingproblem}
    Let $K$ be a totally real number field of degree $d \leq 43$. Assume that $\mathcal{O}^{\vee}_K$ is principal and that $\min_{\gamma \in \mathcal{O}_K^{\vee,+}} \mathrm{Tr}(\gamma)=r_d$, where $r_d$ is given by \eqref{eqr}. If there is a universal $\mathbb{Z}$-form over $K$, then 
    $$|\Delta_K|<|b_{r_d}(2d)(4\pi^2)^d d|^{\frac{2}{3}}.$$
\end{theorem}

\begin{proof}%[\it Proof of Theorem $\ref{theoremliftingproblem}]
The proof of \cite[Theorem 5.1]{kalayatsyna} carries over almost verbatim until the middle of Page 18 of \cite{kalayatsyna}. This is because the assumption that $\min_{\gamma \in \mathcal{O}_K^{\vee,+}} \mathrm{Tr}(\gamma)=r_d$ combined with Theorem \ref{theoremzetafuntion} imply that $\zeta_K(-1)=2^d  b_{r_d}(2d) s_{r_d}^K(2)$.

We also note that if $\alpha=\alpha'\delta \in \mathcal{O}_K^{\vee,+}$, with $\alpha'\in \mathcal{O}_K$, is such that $\mathrm{Tr}(\alpha'\delta)=r_d$, then \cite[Lemma 4.6]{kalayatsyna} implies that $\alpha'$ is a unit $\mathcal{O}_K$ and, hence, $\mathcal{O}_K^{\vee}=(\delta)=(\alpha)$. From this we deduce that $$\sigma((\alpha)(\mathcal{O}_K^{\vee})^{-1})=\sigma(\mathcal{O}_K^{\vee}(\mathcal{O}_K^{\vee})^{-1})=\sigma(\mathcal{O}_K)=1,$$ exactly as in the proof of \cite[Theorem 5.1]{kalayatsyna}.
\end{proof}

\section{Generators for the positive part of ideals}\label{section3}

One idea for extending the proof of \cite[Theorem 5.1]{kalayatsyna} when $\mathcal{O}_K^{\vee,+}$ is not principal is to consider the direct sum of forms $\mathrm{Tr}(\delta_1 Q),\dots, \mathrm{Tr}(\delta_n Q)$ for suitable $\delta_1,\dots, \delta_n \in \mathcal{O}_K^{\vee,+}$, instead of just $\mathrm{Tr}(\delta Q)$. More precisely, one can use the following approach. We can consider  whether there exists a decomposition 
$$\mathcal{O}_K^{\vee,+,0}= \delta_1 \mathcal{O}_K^{+, 0}+\dots+ \delta_n\mathcal{O}_K^{+, 0},$$ where $\mathcal{O}_K^{+,0}=\mathcal{O}_K^+ \cup \{ 0\}$ and $\mathcal{O}_K^{\vee, +,0}=\mathcal{O}_K^{\vee, +} \cup \{ 0\}$, for $\delta_1,\dots,\delta_n \in K^+$ such that the number $n$ only depends on the degree of $K$. If it did exist, then one could employ the same proof strategy as before, but using the form 
$\mathrm{Tr}(\delta_1 Q+\dots+ \delta_n Q)=\mathrm{Tr}(\delta_1 Q)+\dots+\mathrm{Tr}(\delta_n Q)$ (by which we mean the orthogonal sum of quadratic forms) instead of just $\mathrm{Tr}(\delta Q)$ for some $\delta \in \mathcal{O}_K^{\vee, +}$. 

To proceed slightly more generally, let $I$ be a fractional ideal of $\mathcal{O}_K$. We are now interested in answering Question \ref{question1} of the introduction, i.e., we are interested in providing a bound for the number of totally positive generators for $I^{+,0}$. By clearing the denominators of $I$, we can relate the number of totally positive generators for $I^{+,0}$ to the number of totally positive generators of an integral ideal. For this reason, we can work with integral ideals for many of our results. 
However, unless we explicitly state otherwise, by an ideal we always mean a non-zero fractional ideal.

Let now $I$ be a (fractional) ideal and let $$S=\{ \alpha \in I^+ \; : \; \alpha \text{ cannot be written as } \alpha= \beta + \gamma \text{ with } \beta, \gamma \in I^+ \}.$$
We call each element $\alpha$ of this set $I$-indecomposable, generalizing the notion of indecomposables (see the survey \cite{kalasurvey} for more on indecomposables and how they relate to quadratic forms). If we let $S'$ to be a set of representatives of $S$ modulo multiplication by totally positive units of $\mathcal{O}_K$, then we have 
\begin{equation}\label{eq:ideal}
    I^{+,0}= \sum_{\alpha \in S'} \alpha \mathcal{O}_K^{+, 0}.
\end{equation}

We now proceed to obtain bounds on the cardinality of the set $S'$. The following theorem is proved using an adaptation of an argument of Kala--Yatsyna \cite[Theorem 5]{kalayatsynabulletinofthelms}.

\begin{proposition}\label{propnormindecomposables}
    Let $K$ be a totally real number field of discriminant $\Delta_K$ and let $I$ be an integral ideal of $\mathcal{O}_K$. For every element $\alpha \in I^+$ with $\mathrm{N}(\alpha) > \Delta_K \mathrm{N}(I)^2$ there exists $\beta \in I$ such that $\alpha \succ \beta^2$.

    In particular, no such element $\alpha$ is $I$-indecomposable.
\end{proposition}

\begin{proof}
    Assume that the degree of the extension $K/\mathbb{Q}$ is $d$. In the Minkowski space associated to $K$ consider the box defined by $|x_i| \leq \sqrt{\sigma_i (\alpha)} - \epsilon$, where $\sigma_1,\dots,\sigma_d$ are the embeddings of $K$ into $\mathbb{R}$, and $\epsilon > 0$ is small enough so that $$\prod_{i=1}^d \sqrt{\sigma_i(\alpha)}> \sqrt{\Delta_K} \mathrm{N} (I).$$ This is possible because $$\prod_{i=1}^d \sqrt{\sigma_i(\alpha)}= \sqrt{\mathrm{N}(\alpha)} > \sqrt{\Delta_K} \mathrm{N} (I).$$ The volume of the box is bigger than $2^d \sqrt{\Delta_K} \mathrm{N} (I)$. Therefore, since the volume of the fundamental domain $\iota (I)$, where $\iota$ is the Minkowski embedding, is equal to $\sqrt{\Delta_K} \mathrm{N} (I) $ \cite[Proposition (5.2)]{neukirchbook}, we find using, Minkowski's theorem \cite[Page 27]{neukirchbook} we find that there exists a non-zero lattice point in this box. Thus, there exist $\beta \in I$ such that $\sqrt{\sigma_i(\alpha)}> \sigma_i(\beta)$ for $i=1,\dots,d$. This implies that $\sigma_i(\alpha) > \sigma_i(\beta^2)$ for $i=1,\dots,d$, i.e., $\alpha \succ \beta^2$. This proves our proposition.
\end{proof}

The following result will be useful in our work below.

\begin{proposition}\label{propoboundonnumberofreps}
    Let $K$ be a totally real number field of degree $d$. For any positive integer $X$ we let $V_X$ be a set of representatives of classes of elements $\alpha \in \mathcal{O}_K^+$ with $\mathrm{N}(\alpha) \leq X$, up to multiplication by squares of units in $\mathcal{O}_K$. Then $$\# V_X \ll X (\log{(X)})^{d-1}, $$ where the implied constant only depends on $d$.
\end{proposition}
\begin{proof}
    We proceed exactly as in the second part of the proof of \cite[Theorem 6]{kalayatsynabulletinofthelms}. We will omit the details.
\end{proof}

\begin{definition}\label{defkappa}
 Let $K$ be a totally real number field  with ring of integers $\mathcal{O}_K$ and let $I$ be a fractional ideal of $\mathcal{O}_K$. Define $\kappa(I)$ as the smallest number $n$ such that there exist $\alpha_i \in K^+ $ for $i=1,2,\dots,n$ with $$I^{+,0}=\alpha_1 \mathcal{O}_K^{+, 0}+\dots+ \alpha_n\mathcal{O}_K^{+, 0}.$$ Moreover, we define $$\kappa(K)=\max_{I \text{ ideal of } \mathcal{O}_K} \kappa(I).$$
\end{definition}

For example, viewing $\mathcal{O}_K$ as an ideal, we have $\kappa(\mathcal{O}_K)=1$, for $\mathcal{O}_K^{+,0}=1\cdot \mathcal{O}_K^{+, 0}$.

Recall that the narrow class group of a number field $K$ is defined as the quotient $I_K/P_K^+$, where $I_K$ is the group of fractional ideals of $K$ and $P_K^+$ is the subgroup of principal fractional ideals that have a totally positive generator. We denote the narrow class group of $K$ by $Cl^+_K$.

\begin{lemma}\label{kappadoesntdependonclass}
    If $I, J$ are two fractional ideals of $K$ that belong to the same class of $Cl^+_K$, then $\kappa(I)=\kappa(J)$.
\end{lemma}
\begin{proof}
    Since $I$ and $J$ belong to the same class in $Cl^+_K$, we have that $I=\alpha J$ for some $\alpha \in K^+$. Thus if $\{ \alpha_1,\alpha_2,\dots \} $ is a generating set for $J^{+,0}$ (finite or infinite), then $\{ \alpha \alpha_1,\alpha \alpha_2,\dots \} $ is a generating set for $I^{+,0}$. This implies that $\kappa(I) \leq \kappa(J)$. By an entirely similar argument using $\frac{1}{\alpha}$ instead of $\alpha$ we find that $\kappa(J) \leq \kappa(I)$. Therefore, $\kappa(I) = \kappa(J)$. 
\end{proof}

\begin{proposition}\label{prop:kappaI}
    If $I$ is any fractional ideal in a totally real number field $K$, then $\kappa(I)$ is at most the number of classes of $I$-indecomposables modulo multiplication by totally positive units of $\mathcal{O}_K$.    
    In particular, $\kappa(I)$ is finite.
\end{proposition}
\begin{proof}
   There exists $c \in \mathcal{O}_K^+$ such that $cI$ is an integral ideal of $\mathcal{O}_K$. Since $\kappa(I) = \kappa(cI)$ by Lemma \ref{kappadoesntdependonclass}, we can assume without loss of generality that $I$ is an integral ideal. 
   Recall that by \eqref{eq:ideal} we have  $$I^{+,0}= \sum_{\alpha \in S'} \alpha \mathcal{O}_K^{+, 0},$$ where $S'$ is the set of $I$-indecomposables modulo the action by totally positive units. 
   By Proposition \ref{propnormindecomposables}, all elements of $S'$ have bounded norm, and by  Proposition \ref{propoboundonnumberofreps}, there are only finitely many such elements.
\end{proof}

Now we can show that $\kappa(K)$ is finite thanks to the finiteness of the narrow class group.

\begin{theorem}\label{propositionkappa}
For every totally real number field $K$, the number $\kappa(K)$ is finite.
\end{theorem}
\begin{proof}
     Let $I_1,\dots, I_s$ be fractional ideals of $K$ that are representatives of the narrow class group of $K$. It follows immediately from Lemma \ref{kappadoesntdependonclass} that $\kappa(K)= \max \{ \kappa(I_1),\dots,\kappa(I_s) \} $, which proves our theorem.
\end{proof}

\begin{example}
    A totally real number field $K$ has narrow class number 1 if and only if it has class number $1$ and has units of all signatures. 
    In such a case, we can take $I_1=\mathcal{O}_K$ as the only representative in the proof of Theorem \ref{propositionkappa}, and so  $\kappa (K)=\kappa (\mathcal{O}_K)=1.$ 
    For example, we have that $\kappa( \mathbb{Q}(\sqrt{2}))=1$ because $\mathbb{Q}(\sqrt{2})$ has class number $1$ and its fundamental unit is $1+\sqrt{2}$, which has negative norm (see, for example, the LMFDB \cite{lmfdb} database for these basic facts).
\end{example}

\begin{remark}
Let $K$ be a totally real field and let $I$ be an ideal of $\mathcal{O}_K$. It follows from \cite[Lemma 2.1]{fukshanskywang} that $I$ has a $\mathbb{Z}$-basis with basis elements in $I^+$. Let $A=\{a_1,\dots,a_d\} $ be such a basis and consider the set $$S(A)=\{ \lambda_1 a_1+\dots+\lambda_d a_d \: :\;  \lambda_i \in \mathbb{Z}^+ 
\text{ for every } i\}. $$ Since  $S(A)$ is contained $I^+$, one might hope that the above inclusion would be equality. However, the set $I^+ \setminus S(A)$ is always infinite (see \cite[Page 11]{fukshanskywang}).
\end{remark}

Before we focus on the case of real quadratic fields we need to introduce some notation and terminology that will be used. Let $K=\mathbb{Q}(\sqrt{D})$ be a real quadratic number field with ring of integers $\mathcal{O}_K$. Write $\sigma$ for the generator of the Galois group of $K/\mathbb{Q}$; for $\alpha \in K$ we denote its conjugate by $\alpha'=\sigma(\alpha)$. 

We define the $(+,-)$ part of $\mathcal{O}_K$, denoted by $\mathcal{O}_K^{(+,-)}$, as the set consisting of all elements $\alpha \in \mathcal{O}_K$ such that $\alpha > 0$ and $\alpha' <0$. An element of $\alpha \in \mathcal{O}_K^{(+,-)}$ is called a $(+,-)$-indecomposable if it cannot be written as $\alpha= \beta + \gamma $ with $\beta, \gamma \in \mathcal{O}_K^{(+,-)}$.

Let now $$\omega_D= \begin{cases} 
          \sqrt{D} & \text{\ \ if \ \ } D \equiv 2,3 \; (\text{mod } 4), \\
          \frac{1+\sqrt{D}}{2} & \text{\ \ if \ \ }  D \equiv 1 \; (\text{mod } 4). \\
\end{cases}$$ Note that with this convention we have that $\mathcal{O}_K=\mathbb{Z}[\omega_D]$. We also let $$\xi_D=-\omega_D'= \begin{cases} 
          \sqrt{D} & \text{\ \ if \ \ }  D \equiv 2,3 \; (\text{mod } 4), \\
          \frac{\sqrt{D}-1}{2} & \text{\ \ if \ \ }  D \equiv 1 \; (\text{mod } 4). \\
\end{cases}$$ 

In the following proposition, for a principal ideal $I$ of $\mathcal{O}_K$, we relate the number $\kappa(I)$ with the continued fraction expansion of $\xi_D$.

\begin{proposition}\label{propprincipalquadratic}
    Let $K=\mathbb{Q}(\sqrt{D})$ be a real quadratic field and let  $\xi_D=[u_0,\widebar{u_1,\dots,u_s}]$ be the periodic continued fraction expansion of $\xi_D$. 
    Let $I=(\alpha)$ be a principal ideal of $K$ for some $\alpha\in K, \alpha>0$. Then
    $$\kappa(I) \leq \begin{cases} 
          u_1+u_2+u_3+\dots+u_{s-1}+u_s & \text{\ \ if \  }  s \text{\ odd, \ } \\
          u_1+u_3+\dots+u_{s-1} & \text{\ \ if \  }  s \text{\ even and \ } \alpha'>0, \\
          u_2+u_4+\dots+u_{s} & \text{\ \ if \  }  s \text{\ even and \ }\alpha'<0. 
          \end{cases}$$  
\end{proposition}
\begin{proof} 
To prove the result, we will use Proposition \ref{prop:kappaI} thanks to which the number $\kappa(I) $ is smaller or equal to the number of classes of $I$-indecomposables modulo multiplication by totally positive units.

Let $\ve$  be the fundamental unit of $K$.
Recall that the period length $s$ is odd if and only if the norm $N(\ve)=-1$.
For convenience, assume that the continued fraction is $\xi_D=[u_0,\widebar{u_1,\dots,u_s}]=[u_0,u_1,\dots,u_s,u_{s+1},\dots]$, i.e., we are letting $u_{ts+i}=u_i$ for each $t,i$.

If $I$ has a totally positive generator, then $\kappa(I)=\kappa(\co_K)$ by Lemma \ref{kappadoesntdependonclass}. In particular, this is always the case when $s$ is odd. If $I=(\alpha)$ for $\alpha>0,\alpha'<0$, then $$I^{+}=\{\alpha\beta\in\co_K^{+}\: : \; \beta\in\co_K\}=\{\alpha\beta\: : \; \beta\in\co_K^{(+,-)}\}\simeq \co_K^{(+,-)}$$
(here, the last isomorphism $\simeq$ is an isomorphism of additive semigroups). 
Thus the $I$-indecomposables are in a bijection with $(+,-)$-indecomposables (and likewise for their classes up to unit multiplication). Thus in both cases, we are interested in the number of unit classes of indecomposables either in $\co_K$ or in $\mathcal{O}_K^{(+,-)}$.

Luckily, in both of these cases, the indecomposables have a nice characterization as the upper (or lower) semiconvergents to $\xi_D$. Specifically, let $\frac{p_i}{q_i}=[u_0,\dots,u_i]$ for coprime positive positive integers $p_i, q_i$ and $i \geq 0$. We also write $p_{-1}=1$ and $q_{-1}=0$. For every $i \geq -1$ we call the algebraic integers $\alpha_i=p_i-q_i\omega_D'$ the convergents. We also define the semiconvergents $\alpha_{i,r}=\alpha_i+r\alpha_{i+1}$ for $i \geq -1$ and $0 \leq r < u_{i+2}$.
Then, up to conjugation $\sigma$, the elements 
$\alpha_{i,r}$  with odd $i\geq -1$ are exactly all the indecomposables in $\co_K$, and these elements with even $i\geq 0$ are exactly the $(+,-)$-indecomposables 
(see \cite[Subsection 2.1]{BK}, \cite[Theorem 2]{dressscharlau}, \cite[\S 16]{perronbook}; note that the first two cited references deal only with the case of odd $i$ and indecomposables in $\co_K$, but the case of $(+,-)$-indecomposables is completely analogous).

As $\ve=\alpha_{s-1}$ and $\ve\alpha_{i,r}=\alpha_{i+s,r}$, the preceding characterization immediately implies that the numbers of unit classes of indecomposables are as follows (cf. \cite[Subsection 2.1]{BK}):

CASE 1: Assume that $s$ odd. Then we do not need to distinguish indecomposables in $\co_K$ and in $\mathcal{O}_K^{(+,-)}$ (because multiplication by $\ve$ turns the former into the latter).
Since we are interested in the number of indecomposables up to multiplication of totally positive units (that are exactly $\ve^{2n}$ for $n\in\Z$, and $\ve^2=\alpha_{2s-1}$), we need to consider odd indices $i$ in the range $-1\leq i< 2s-1$. These give us 
\begin{align*}
&u_1+u_3+\dots+u_{s}+u_{s+2}+\dots+u_{2s-1}=\\
&u_1+u_3+\dots+u_{s}+u_{2}\ \ \ +\dots+u_{s-1}\ =\\
&u_1+u_2+u_3+\dots+u_{s-1}+u_s.
\end{align*}

CASE 2: Assume that $s$ even and that $\alpha'>0$. We need to count the number of indecomposables in $\co_K$. We have that $\ve=\alpha_{s-1}$ is the fundamental totally positive unit, and so the classes of indecomposables are given by $\alpha_{i,r}$ for odd $i$ in the range $-1\leq i< s$. There are clearly $u_1+u_3+\dots+u_{s-1}$ of them.

CASE 3: Assume that $s$ even and that $\alpha'<0$. We need to count the number of indecomposables in $\mathcal{O}_K^{(+,-)}$. We only need to consider the even indices $i$ in the range $-1\leq i< s$, obtaining $u_2+u_4+\dots+u_{s}$. 
    \end{proof}

\begin{example}
    In this example we show, using ideas from Proposition \ref{propprincipalquadratic} above, that $\kappa(\mathbb{Q}(\sqrt{3}))=2$. The class group of $\mathbb{Q}(\sqrt{3})$ is trivial and the fundamental unit is $2+\sqrt{3}$, which has positive norm (see, for example, the LMFDB \cite{lmfdb} database for these basic facts). Let now $I$ be any ideal of $\mathbb{Q}(\sqrt{3})$. Then $I$ is principal and is either generated by a totally positive element or by an element in $\mathcal{O}_{\mathbb{Q}(\sqrt{3})}^{(+,-)}$. In the first case, we have that $\kappa(I)=1$. So assume from now on that $I$ is generated by an element in $\mathcal{O}_{\mathbb{Q}(\sqrt{3})}^{(+,-)}$. The continued fraction expansion of $\sqrt{3}$ is $[1,\widebar{1,2}]$. Using Proposition \ref{propprincipalquadratic} we find that $\kappa(I) \leq 2$. Thus $\kappa(\mathbb{Q}(\sqrt{3})) \leq 2$. 
    
    We now show that equality is achieved, i.e., that $\kappa(\mathbb{Q}(\sqrt{3})) = 2$. Let $I=(\alpha)$ be an ideal of $\mathbb{Q}(\sqrt{3})$ with $\alpha \in \mathcal{O}_{\mathbb{Q}(\sqrt{3})}^{(+,-)}$. Assume that $I^+=\alpha^+ \mathcal{O}_{\mathbb{Q}(\sqrt{3})}^{+,0}$ with $\alpha^+ \in \mathcal{O}_{\mathbb{Q}(\sqrt{3})}^{+}$, and we will find a contradiction. Since $\alpha^+ \in I$, there exists $\beta \in \mathcal{O}_{\mathbb{Q}(\sqrt{3})}^{(+,-)}$ such that $\alpha^+=\alpha \beta$. Consider now the ideal $I'=(\alpha^+)$. It follows that $I' \subsetneq  I $ because $\beta$ is not a unit due to the fact that the fundamental unit is totally positive. On the other hand, if we denote by $\sigma$ the generator of the Galois group of the extension $\mathbb{Q}(\sqrt{3})/\mathbb{Q}$, since $\alpha \in \mathcal{O}_{\mathbb{Q}(\sqrt{3})}^{(+,-)}$, then $-\alpha \sigma (\alpha) \in \mathbb{Z}^+ \cap I$ and $\alpha -\alpha \sigma (\alpha) \in \mathcal{O}_{\mathbb{Q}(\sqrt{3})}^{+} \cap I$. Therefore $\alpha = (\alpha -\alpha \sigma (\alpha)) - (-\alpha \sigma (\alpha)) \in I'$, which implies that $I \subseteq I'$. This is a contradiction. Therefore, we proved that $\kappa(I)=2$.
\end{example}

Note that $u_s=2u_0$ if $D\equiv 2,3\pmod 4$ and $u_s=2u_0-1$ if $D\equiv 1\pmod 4$, and $u_0=\lfloor\xi_D\rfloor$ in both cases. Thus $u_s\gg \sqrt D$, and so the bound for $\kappa(I)$ from Proposition \ref{propprincipalquadratic} is growing as $\sqrt D$, except for the case of even $s$ and totally positive $\alpha$. 
However, it is conjectured that infinitely many real quadratic fields $\Q(\sqrt D)$ have narrow class number 1 (then $s$ is odd), in which case we know that $\kappa(K)=1$. Thus we expect the bound from Proposition \ref{propprincipalquadratic} to be very bad for infinitely many $D$. While it may be possible to establish a similar bound also for non-principal ideals (in terms of continued fractions of different elements), 
determining the correct order of magnitude of $\kappa(K)$ in general seems to be hard.

\bibliographystyle{plain}
\bibliography{bibliography.bib}

\end{document}